\documentclass{article}
\usepackage{graphicx} 
\usepackage{mathtools}
\usepackage{bm}
\usepackage{amssymb}
\usepackage{amsmath}
\usepackage{amsthm}
\usepackage{comment}
\usepackage{dsfont}

\usepackage{color} 
\usepackage[utf8]{inputenc}
\usepackage[T1]{fontenc}
\usepackage[colorlinks=true]{hyperref}

\usepackage[numbers]{natbib}
 
\usepackage[dvipsnames]{xcolor}

\usepackage{todonotes}
\usepackage{multirow}

\newcommand{\dd}{\mathrm{d}}
\newcommand{\sdd}{\,\dd}
\newcommand{\abs}[1]{\left\lvert #1 \right\rvert}
\newcommand{\E}{\mathbb{E}}
\newcommand{\R}{\mathbb{R}}
\newcommand{\id}{\textup{Id}}
\newcommand{\diag}{\textup{diag}}

\theoremstyle{plain}
\newtheorem{theorem}{Theorem}[section]
\newtheorem{lemma}[theorem]{Lemma}

\newtheorem{corollary}[theorem]{Corollary}

\theoremstyle{definition}
\newtheorem{definition}[theorem]{Definition}

\newtheorem{example}[theorem]{Example}
\newtheorem{remark}[theorem]{Remark}

\usepackage{authblk}
\mathtoolsset{showonlyrefs}
\title{State spaces of multifactor approximations of nonnegative Volterra processes}

\author[1]{Eduardo Abi Jaber\thanks{eduardo.abi-jaber@polytechnique.edu.}}
\author[2]{Christian Bayer\thanks{christian.bayer@wias-berlin.de}}
\author[2]{Simon Breneis\thanks{simon.breneis@wias-berlin.de}}
\affil[1]{École Polytechnique, CMAP}
\affil[2]{Weierstrass Institute Berlin}

\date{\today}

\begin{document}

\maketitle

\begin{abstract}
	We show that the state spaces of multifactor Markovian processes, coming from approximations of nonnegative Volterra processes, are given by explicit linear transformation of the nonnegative orthant. We demonstrate the usefulness of this result for applications, including simulation schemes and PDE methods for nonnegative Volterra processes.
\end{abstract}

\begin{description}
\item[Mathematics Subject Classification (2020):]  60H20, 91-10, 	91-08
\item[Keywords:]  Stochastic Volterra equations, Markovian approximations, {rough Heston model}, Viability and Invariance, Simulation, PDEs
\end{description}


\section{Introduction}

Multifactor Markovian approximations for fractional Brownian motion were initially introduced by \citet{carmona1998fractional} and revisited more recently by \citet{abi2019multifactor} in the context of nonnegative stochastic Volterra equations, motivated by rough and Volterra Heston models of \citet{el2019characteristic} and \citet*{abi2019affine}.  Since then, substantial literature has emerged on such multifactor processes for numerical approximation methods (\citet*{alfonsi2021approximation, bayer2023markovian, bayer2023++Weak, bayer2023++++Simulation, chevalier2022american,  harms2019strong}), deep learning approaches (\citet{papapantoleon2024time}), modeling (\citet{abi2019lifting}), and optimal control (\citet*{abi2021linear}).  
It should be noted that such approximations are also heavily used in physics, chemistry and other fields, see, e.g., \citet*{baczewski2013numerical, bochud2007optimal}.

The starting point is a nonnegative  solution  to  the stochastic Volterra equation
\begin{equation}\label{eqn:SVEExponentialKernel}
	Y_t = Y_0 + \int_0^t K(t-s) b(Y_s) \sdd s + \int_0^t K(t-s) \sigma(Y_s) \sdd W_s,
\end{equation}
where the  kernel $K$  is (approximated by) a weighted sum of exponentials of the form 
\begin{align}\label{eq:multiexpkernel}
	K(t)=  \sum_{i=1}^N w_i e^{-x_it}
\end{align}
with positive nodes $\bm x = (x_i)_{i=1,\ldots,N}$ and weights $\bm w = (w_i)_{i=1,\ldots,N}$.
Such series in terms of exponential functions are sometimes known as \emph{Prony series}.
The coefficients 
$b,\sigma:\R\to\R$  are continuous and satisfy the boundary conditions
\begin{equation}\label{eqn:BoundaryCondition}
	b(0) \ge 0\quad\text{and}\quad \sigma(0) = 0,
\end{equation}
to ensure that the process $Y$ remains nonnegative for any $Y_0\geq 0$.

Then, the nonnegative Volterra process $Y$ can be written in the form $Y_t = \sum_{i=1}^N w_i Y^{(i)}_t$, where $\bm Y = (Y^{(i)})_{i=1,\ldots,N}$ is the solution to the $N$-dimensional Stochastic Differential Equation (SDE)
{\begin{equation}
		\dd Y_t^{(i)} = -x_i\left(Y_t^{(i)} - Y_0^{(i)}\right)\sdd t + b(Y_t)\sdd t + \sigma({Y_t}) \sdd W_t,
	\end{equation}
with initial values $Y^{(i)}_0$, such that $\sum_{i=1}^N w_i Y^{(i)}_0 = Y_0$. More generally, we will consider the  $N$-dimensional  SDE 
\begin{equation}\label{eqn:VolMarkov}
		\dd Y_t^{(i)} = -x_i\left(Y_t^{(i)} - y_0^{(i)}\right)\sdd t + b(Y_t)\sdd t + \sigma({Y_t}) \sdd W_t,
	\end{equation}
  with $y^{(i)}_0 \in \R$, which are often, but not necessarily chosen to coincide  with the initial values $Y^{(i)}_0$,  for $i=1,\dots,N$.}

The aim of the paper is to determine a state space of the multifactor Markovian process $\bm Y$.  That is, we want to determine a set $\mathcal{D}\subseteq\R^N$ such that for every starting value $\bm Y_0\in\mathcal{D}$, there exists a  $\mathcal D$-valued  solution $\bm Y$ to  \eqref{eqn:VolMarkov}, that is  $\bm Y_t \in\mathcal{D}$ for all $t\ge 0$ almost surely.

Beyond the mathematical importance of defining the state space of the Markovian process $\bm Y$, the knowledge of the state space is crucial for several practical applications, some of which  have been considered so far:
\begin{itemize}
		\item \textbf{Modeling and Calibration:} The multifactor model \eqref{eqn:VolMarkov} can serve as a model in its own right (and not solely as an approximation of Volterra models), usually for stochastic volatility factors, as seen in the lifted Heston model of \citet{abi2019lifting}. Here, the knowledge of the state space is crucial for calibrating the initial values of \( \bm Y_0 \) directly to market data.
		
	\item \textbf{Simulation accuracy:} The identification of a valid state space allows for more precise  simulation schemes for $\bm Y$. For instance, \citet{bayer2023++++Simulation} generalized a simulation scheme for the square-root process of \citet{lileika2021second} to simulate paths from the dynamics in \eqref{eqn:VolMarkov} with $\sigma(z)=\sqrt{z}$. However, the authors were not able to show that their simulation scheme is well-defined because they did not know the state space. Indeed, when simulating from \eqref{eqn:VolMarkov}, great care has to be taken to ensure that the aggregated process $Y$ does not become negative, or if it does, one has to determine how to proceed with the square root term $\sqrt{Y}$.

	\item \textbf{Efficient Domain Meshing for PDE Solutions:} In \citet{papapantoleon2024time}, for example, the authors use a rejection algorithm that discards simulations failing to satisfy the condition \( \sum_i w_i Y^{(i)}_0 \geq 0 \),  this approach is not precise  and  becomes inefficient in high dimensions.

\end{itemize}

For all these reasons, the geometry of the state space for multifactor processes $\bm Y$ quickly became a central focus in the associated literature for nonnegative Volterra processes.  
For the lifted Heston model of \citet{abi2019lifting}, simulations demonstrated that while individual processes $Y^{(i)}$ might take negative values, their aggregated sum $Y$ remains nonnegative. In this context, two key papers provided abstract characterizations of possible state spaces: one based on the resolvent of the first kind of the kernel by \citet{abi2019markovian}, and the other using the  resolvent of the second kind of the kernel by \citet{cuchiero2020generalized} (we refer to  Section \ref{A:link} for further details on these state spaces and their  connections). While these spaces are valid for a wide range of kernels, they remain somewhat abstract and challenging to make explicit, making it almost impossible to determine the good conditions on the initial values $\bm Y \ge 0$ of the process \( \boldsymbol{Y} \). Finally,  {using the simulation algorithm for the rough Heston model due to \citet*{bayer2023++++Simulation}, we visually represent the process's domain by a sample plot, see Figure \ref{fig:SamplesOfTwoDimensionalVolatilityProcess} -- replicating a similar plot already presented in that paper.}
One can clearly recognize that the sample paths do not only seem to lie in the half-plane $Y = w_1 Y^{(1)} + w_2 Y^{(2)} \ge 0$ marked with the {down-ward oriented} black line, but in an even smaller cone seemingly below the {upward oriented} line $\{\bm y \in \R^2: y_1 \geq y_2 \}$.

\begin{figure}[h]
	\centering
	\includegraphics[width=0.7\linewidth]{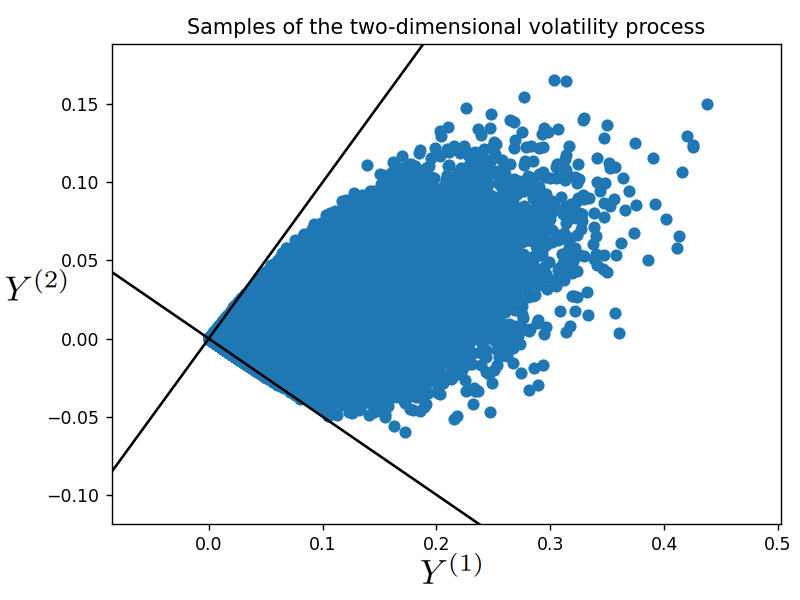}
	\caption{Samples of the two-dimensional process $(Y^{(1)}, Y^{(2)})$ using $10^5$ sample paths on a time grid with $M=1000$ time steps. Plotted are all the points $\bm Y_{t_i}$ for every time step $t_i$, $i=0,\dots,1000$, and all the $10^5$ samples. The {decreasing} black line is the line where the aggregated process $\bm w_1 Y^{(1)} + w_2 Y^{(2)} = 0$, and the aggregated process is positive above that line. The nearly orthogonal second line cuts out a cone, which seems to give the actual domain of the process. Note that no samples actually lie outside the cone inscribed in the figure.}
	\label{fig:SamplesOfTwoDimensionalVolatilityProcess}
\end{figure}

\paragraph{Main contributions.} The main question we are interested in  can be summarized as follows:
\begin{center}
\textit{What constitutes a suitable state space for the multifactor Markovian process $\bm Y$ in \eqref{eqn:VolMarkov}?}
\end{center}

Our main results in Theorem~\ref{thm:SVEDomainTheorem} and Corollary~\ref{cor:SVEDomainTheoremArbitraryy0} establish that this state space can be represented as a linear  transformation of \( \mathbb{R}^N_+ \), and we provide an explicit form for this transformation.

As a first application of this result, we prove that the weak simulation scheme for the rough Heston process proposed by \citet*{bayer2023++++Simulation} is well-defined in the sense that the variance process always stays nonnegative, see Section~\ref{sec:WellDefinitenessOfWeakScheme}.
In Section~\ref{sec:solving_pdes}, we derive the corresponding pricing PDE on the transformed domain $\mathbb{R}^N_+$, and solve it numerically by the finite element method after truncation of the domain. Naturally, knowing the PDE's precise  domain is crucial for accurate numerical approximations.
Finally, in Section~\ref{A:link}, we show how the explicit formula for the domain compares with general, abstract characterizations given in the literature  by \citet*{abi2019markovian} and \citet*{cuchiero2020generalized}.

\paragraph{Notation and Conventions.} We denote by $\bm e_i$  the $i$-th unit vector, which has a $1$ in the $i$-th component, and 0 in every other component, $\bm 1 \coloneqq (1,1,\dots,1)^{\top}$ is the vector with 1 in every component, $\id$ is the identity matrix, $\diag(\bm a)$ for a vector $\bm a \in\R^N$ is the diagonal matrix with entries $\bm a$ in the diagonal, and $\overline w \coloneqq \bm 1^\top \bm w = \sum_{i=1}^N w_i.$ Throughout, italic letters $a$ denote real numbers and bold letters $\bm a$ denote vectors, where we write $\bm a = (a_i)_{i=1}^N$ for the components of $\bm a$. An exception are stochastic processes, where components are denoted by $\bm Y_t = (Y_t^{(i)})_{i=1}^N$ (due to the time variable in the subscript).

\paragraph{Acknowledgments}
EAJ is grateful for the financial support from the Chaires FiME-FDD and Financial Risks at Ecole Polytechnique.
CB and SB gratefully acknowledge the support by the IRTG 2544 ``Stochastic Analysis in Interaction''. CB also acknowledges support from DFG CRC/TRR 388 ``Rough Analysis, Stochastic Dynamics and Related Fields'', Project B02.

\section{State spaces of the multifactor Markovian process}\label{sec:StateSpaceMarkovianApproximation}

Fix $N\geq 1$.  We consider the  $N$-dimensional SDE
\begin{equation}\label{eqn:SVEMarkovCompact}
\dd \bm Y_t = -\diag(\bm x) \left(\bm Y_t-\bm y_0\right)\sdd t + b(\bm w^\top \bm Y_t)\bm 1\sdd t + \sigma(\bm w^{\top}\bm Y_t)\bm 1 \sdd W_t, 
\end{equation}
where $b,\sigma:\R\to\R$  are continuous, satisfy the linear growth condition
\begin{equation}\label{eqn:LinearGrowthCondition}
	\abs{b(y)}\lor\abs{\sigma(y)} \le C\left(1 + \abs{y}\right),\quad y\in\R,
\end{equation}
and the boundary conditions \eqref{eqn:BoundaryCondition}. 
The speeds of mean-reversion $\bm x = (x_i)_{i=1, \ldots, N}$ are positive and ordered, i.e.~$0 < x_1 \le x_2 \le \dots \le x_N$, the weights $\bm w = (w_i)_{i=1, \ldots, N}$ are  positive, $W$ is a one-dimensional Brownian motion, and where $\bm Y_0,\bm y_0\in\R^N$ may be different.  This corresponds to \eqref{eqn:VolMarkov} written in vector form. 

The aim of this section is to determine a state space of the multifactor process  $\bm Y$. That is, we want to determine a set $\mathcal{D}\subseteq\R^N$ such that for every starting value $\bm Y_0\in\mathcal{D}$, there exists a  $\mathcal D$-valued weak solution $\bm Y$ to  \eqref{eqn:SVEMarkovCompact}, that is  $\bm Y_t \in\mathcal{D}$ for all $t\ge 0$ almost surely. In particular, the domain $\mathcal{D}$ should be a subset of the half-plane $\{\bm y\in\R^N\colon \bm w^\top \bm y \ge 0\}$ to ensure non-negativity of the aggregated weighted  process $Y \coloneqq \bm w^\top \bm Y$, which for the specific case $\bm y_0 = \bm Y_0$ would correspond to the Volterra process \eqref{eqn:SVEExponentialKernel} for the weighted sum of exponential kernel $K$ given in \eqref{eq:multiexpkernel}. As illustrated in Figure~\ref{fig:SamplesOfTwoDimensionalVolatilityProcess}, and following the abstract characterizations of  domains of (possibly infinite-dimensional) lifts of nonnegative Volterra processes in \citet{abi2019markovian} and \citet{cuchiero2020generalized},  we suspect the domain $\mathcal{D}$ to be a cone.

\subsection{Main result}

 We prove that the domain $\mathcal D$ is a cone characterized by the set $\mathcal Q$ of admissible matrices, which is defined  as follows.

\begin{definition}\label{def:AdmissibleMatrix}
    A matrix $Q\in\R^{N\times N}$ is called admissible if it satisfies the following assumptions:
    \begin{enumerate}
        \item $Q$ is invertible,
        \item $\bm e_N^\top Q = \bm w^\top $,
        \item $Q\bm 1 = \overline{w}\bm e_N$, where $\overline{w}\coloneqq \bm w^\top \bm 1$,
        \item $(Q\diag(\bm x)Q^{-1})_{i,j} \le 0$ for $i,j\in\{1,\dots,N\}$ with $i\neq j$.
    \end{enumerate}
We denote by $\mathcal Q$ the set of all admissible matrices. 
\end{definition}

Before stating our main theorem, we first show that the set of admissible matrices $\mathcal Q$ is nonempty by providing an explicit example of an admissible matrix.   However, $\mathcal Q$ is not  reduced to a singleton as shown in Example~\ref{ex:InvariantDomainDimension2} below.

\begin{theorem}\label{thm:SpecificQ}
	The matrix $Q=(q_{i,j})_{i,j=1,\ldots N}\in\R^{N\times N}$ given by $$q_{i,j} = w_j,\quad j\le i,\qquad q_{i,i+1}=-\sum_{j=1}^i w_j,\quad i=1,\dots,N-1,$$ and zeros elsewhere is admissible. In particular, $\mathcal Q$ is nonempty. 
\end{theorem}

\begin{proof}
	The proof is given in Section~\ref{S:proofnonempty}.
\end{proof}

We are now ready to state our main theorem that gives state spaces of the multifactor Markovian process \eqref{eqn:SVEMarkovCompact}.

\begin{theorem}\label{thm:SVEDomainTheorem}
    Let $b,\sigma:\R\to \R$ be continuous functions satisfying the linear growth condition \eqref{eqn:LinearGrowthCondition} and the boundary conditions \eqref{eqn:BoundaryCondition}. Let $Q \in \mathcal Q$ be an admissible matrix in the sense of Definition~\ref{def:AdmissibleMatrix} and suppose that
    \begin{equation}\label{eqn:InitialConditionsCondition}
    \bm y_0 = \mu\diag(\bm x)^{-1}\bm 1\quad \text{for some }\mu\ge 0.
    \end{equation}
Set $\mathcal D =Q^{-1}\R_+^N$. Then, for each $\bm Y_0 \in \mathcal D$, there exists a $\mathcal D$-valued weak solution $\bm Y$ to  \eqref{eqn:SVEMarkovCompact}.
\end{theorem}

\begin{proof}
	The proof is given in Section~\ref{S:proofDomainInvariance}.
\end{proof}

\begin{example} For the admissible matrix $Q$ given in Theorem~\ref{thm:SpecificQ},  
	the set  $\mathcal D = Q^{-1} \R_+^N$ corresponds to the set of $\bm y\in \R^N$ such that $\bm w^\top \bm y \ge 0$ and $$\sum_{j=1}^i w_j y_j \ge \sum_{j=1}^i w_j y_{i+1}\qquad\textup{for}\qquad i = 1,\dots,N-1.$$
\end{example}

\begin{remark}
	The domain $Q^{-1}\R_+^N$ is not unique, see  Appendix~\ref{sec:InvariantDomainDimension3}.
\end{remark}

    In practice, for instance for the multifactor approximations of Volterra processes, we often have $\bm Y_0 = \bm y_0$. Hence, it may be interesting to know whether $\bm y_0$ as given in \eqref{eqn:InitialConditionsCondition} is in $Q^{-1}\R_+^N$. We treat this question in a slightly more general context in the following lemma.
    
    \begin{lemma}\label{lem:MMatricesAreFun}
    Let $Q$ be an admissible matrix and let $\bm y_0\in\R^N$ satisfy \eqref{eqn:InitialConditionsCondition}. Then $\bm y_0\in Q^{-1}\R_+^N.$	
    \end{lemma}
    
 \begin{proof}
	We prove the equivalent statement $Q\bm y_0\in\R_+^N$. First, note that $$Q\bm y_0 = \mu Q\diag(\bm x)^{-1}\bm 1 = \mu Q\diag(\bm x)^{-1}Q^{-1}Q\bm 1 = \mu\overline{w} Q\diag(\bm x)^{-1}Q^{-1}\bm e_N.$$
	
	Recall that in linear algebra, an M-matrix is a square matrix with non-positive off-diagonal entries and with eigenvalues whose real parts are non-negative. Clearly, the matrix $Q\diag(\bm x)Q^{-1}$ is an M-matrix: The non-positivity of the off-diagonal entries holds by assumption, and its eigenvalues are (real and) positive, since it is just the matrix $\diag(\bm x)$ written in a different basis. Now it is well-known that the inverse of an M-matrix has non-negative entries (in fact, this property characterizes M-matrices). Therefore, $$\left(Q\diag(\bm x)Q^{-1}\right)^{-1} = Q\diag(\bm x)^{-1}Q^{-1}$$ has only non-negative entries. In particular, $$Q\bm y_0 = \mu\overline{w}Q\diag(\bm x)^{-1}Q^{-1}\bm e_N\in\R_+^N,$$ proving the lemma.
\end{proof}

    We remark that condition \eqref{eqn:InitialConditionsCondition} on $\bm y_0$ can easily be dropped by using affine transformations of $\R^N_+$ instead of linear ones. We give the specific details in the next corollary.

\begin{corollary}\label{cor:SVEDomainTheoremArbitraryy0}
    Let $b,\sigma:\R\to \R$ be continuous functions satisfying the linear growth condition \eqref{eqn:LinearGrowthCondition} and the boundary conditions \eqref{eqn:BoundaryCondition}. Let $Q \in \mathcal Q$ be an admissible matrix in the sense of Definition~\ref{def:AdmissibleMatrix}, and assume that $\bm w^\top  \bm y_0 \ge 0$. Let $\widetilde{\bm y}_0$ be chosen according to \eqref{eqn:InitialConditionsCondition} with $\bm w^\top \widetilde{\bm y}_0 = \bm w^\top \bm y_0.$ Define the set
    \begin{align}\label{eq:Dlinear}
    	\mathcal{D}= Q^{-1}\R_+^N + (\bm y_0 - \widetilde{\bm y}_0).
    \end{align}
     Then, for any  $\bm Y_0 \in \mathcal{D}$, there exists a $\mathcal D$-valued weak solution $\bm Y$ to \eqref{eqn:SVEMarkovCompact}.
\end{corollary}

\begin{proof}
    Denote by $\widetilde{\bm Y}$ a weak solution to $$\dd \widetilde{\bm Y}_t = -\diag(\bm x)\left(\widetilde{\bm Y}_t - \widetilde{\bm y}_0\right)\sdd t + b(\bm w^\top  \widetilde{\bm Y}_t)\bm 1\sdd t + \sigma(\bm w^\top \widetilde{\bm Y}_t)\bm 1\sdd W_t$$ with initial condition $\widetilde{\bm Y}_0 \coloneqq \bm Y_0 + \widetilde{\bm y}_0 - \bm y_0.$ Note that $\widetilde{\bm Y}_0\in Q^{-1}\R_+^N$, so Theorem \ref{thm:SVEDomainTheorem}  imply the existence of such a solution $\widetilde{\bm Y}$ that stays in $Q^{-1}\R_+^N$.    Define the process $\bm R \coloneqq \widetilde{\bm Y} + \bm y_0 - \widetilde{\bm y}_0$ and note that $\bm w^\top \bm R = \bm w^\top \widetilde{\bm Y}$. Thus, $\bm R$ satisfies $\bm R_0 = \bm Y_0$ and
    \begin{align*}
    \dd \bm R_t &= -\diag(\bm x)\left(\bm R_t - \bm y_0\right)\sdd t + b(\bm w^\top  \bm R_t)\bm 1\sdd t + \sigma(\bm w^\top \bm R_t)\bm 1\sdd W_t.
    \end{align*}
But this means that $\bm R$ is a solution to \eqref{eqn:SVEMarkovCompact} that stays in $\mathcal{D}$. The corollary follows immediately.
\end{proof}

We now give the specific result for the multifactor square-root process. 

\begin{example}\label{thm:rHestonDomainTheorem}
    Consider the multifactor square-root model
    \begin{equation}\label{eqn:RHestonVolMarkovCompact}
        \dd \bm V^N_t =  - \diag(\bm x)\left(\bm V^N_t - \bm v_0\right)\sdd t + \left(\theta - \lambda \bm w^\top \bm V^N_t\right) \bm 1 \sdd t + \nu \sqrt{\bm w^\top  \bm V^N_t} \bm 1 \sdd W_t.
    \end{equation}
    Let $Q$ be an admissible matrix, and assume that $\bm w^\top \bm v_0 \ge 0$. Then,  $$\mathcal{D}= Q^{-1}\R_+^N + \left(\bm v_0 - \frac{\bm w^\top \bm v_0}{\bm w^\top \diag(\bm x)^{-1}\bm 1}\diag(\bm x)^{-1}\bm 1\right).$$ 
\end{example}


\subsection{Link with nonnegative Volterra processes}

As an application of our result, one can obtain the existence of nonnegative solutions to Volterra equations with kernels of the form \eqref{eqn:SVEExponentialKernel}. We note that such existence can be obtained by working directly on the level of the Volterra equation as done in  \citet*[Theorem 3.6 and Example 3.7]{abi2019affine}. Here, our result provides another alternative as illustrated in the following corollary. 

\begin{corollary}
	 Let $b,\sigma:\R\to \R$ be continuous functions satisfying the linear growth condition \eqref{eqn:LinearGrowthCondition} and the boundary conditions \eqref{eqn:BoundaryCondition}. Let  the kernel $K$  be given by  a weighted sum of exponentials as in \eqref{eq:multiexpkernel}.  Then, for each $Y_0\geq 0$,  the stochastic Volterra equation \eqref{eqn:SVEExponentialKernel}
	 admits a nonnegative weak solution $Y$. 
\end{corollary}

\begin{proof}
	Fix $Y_0 \geq 0$ and let  $ \bm y_0 \in \R^N$ be such that $\bm w^\top \bm y_0 = Y_0$.  Let $\widetilde{\bm y}_0$ be chosen according to \eqref{eqn:InitialConditionsCondition} with $\bm w^\top \widetilde{\bm y}_0 = \bm w^\top \bm y_0$.  Let $Q\in \mathcal Q$ be an admissible matrix, for instance given by Theorem~\ref{thm:SpecificQ}. Then, it follows from Lemma~\ref{lem:MMatricesAreFun} that $ \widetilde{\bm y}_0\in Q^{-1}\R$. Hence, $\bm y_0 =  \widetilde{\bm y}_0  + (\bm y_0 - \widetilde{\bm y}_0) \in \mathcal D$, with $\mathcal D$ given by \eqref{eq:Dlinear}. An application of Corollary~\ref{cor:SVEDomainTheoremArbitraryy0}, with the starting value  $Y_0= \bm y_0 \in \mathcal D$, yields the existence of a $\mathcal D$-valued weak solution $\bm Y$ to the equation \eqref{eqn:SVEMarkovCompact}. Thanks to the variation of constants formula, we can re-write the equation in the form 
	\begin{equation}
		\bm Y_t =  \bm y_0 + \int_0^t  \exp(-\diag(\bm x)(t-s)) \bm 1\left( b(\bm w^\top \bm Y_s) \sdd s  +   \sigma(\bm w^{\top}\bm Y_s) \sdd W_s\right), 
	\end{equation}
so that the process $Y$ defined by $Y = \bm w^\top \bm Y$ solves the equation 
\begin{align*}
	Y_t = Y_0 + \int_0^t \sum_{i=1}^N w_i e^{- x_i(t-s)} \left(b(Y_s) \sdd s + \sigma (Y_s) \sdd W_s\right),
\end{align*}
which is precisely the Volterra equation \eqref{eqn:SVEExponentialKernel} with the kernel $K$ given by \eqref{eq:multiexpkernel}. It remains to argue that, for all $t\geq 0$,  $Y_t$ remains nonnegative by using the fact that $\bm Y_t \in \mathcal D$. Indeed, using Condition 2 of Definition~\ref{def:AdmissibleMatrix} and the fact that  $\bm w^\top \widetilde{\bm y}_0 = \bm w^\top \bm y_0$, we obtain that  
$$ \bm w^\top \mathcal D = \bm e_N^\top Q Q^{-1}\R_+^N + (\bm w^\top \bm y_0 - \bm w^\top \widetilde{\bm y}_0) = \bm e_N^\top \R_+^N = \R_+. $$ 
Hence, for all $t\geq 0$, $Y_t = \bm w^\top \bm Y_t \in \bm w^\top \mathcal D = \R_+ $, which ends the proof. 
\end{proof}

\subsection{On admissible matrices for $N \in \{2,3\}$}

In this section we give examples of admissible matrices.

\begin{example}\label{ex:InvariantDomainDimension2}
    In the case $N=2$, the Conditions 2 and 3 of Definition~\ref{def:AdmissibleMatrix} imply that we are looking for a matrix of the form $$Q = \begin{pmatrix}
        q & -q\\
        w_1 & w_2
    \end{pmatrix},$$
for some $q\neq 0$ to ensure invertibility. 
    Then, $$Q\diag(\bm x) Q^{-1} = \frac{1}{\overline w}\begin{pmatrix}
        w_1x_2 + w_2x_1 & (x_1-x_2)q\\
        w_1w_2(x_1-x_2)q^{-1} & w_1x_1 + w_2x_2
    \end{pmatrix}.$$ Since $x_1 \le x_2$, the last condition in Definition \ref{def:AdmissibleMatrix} is satisfied for any $q > 0$, and indeed, the domain $Q^{-1}\R_+^2$ is independent of the precise choice of $q$ given by
\begin{equation}\label{eqn:DomainEquationDimension2}
        \mathcal{D} = \left\{\bm y\in\R_+^2\colon \bm w^\top  \bm y \ge 0,\ y_1 \ge y_2\right\}.
    \end{equation}
    For the case of the multifactor square-root process \eqref{eqn:RHestonVolMarkovCompact}, the resulting sample paths of $\bm V^2$ and $\bm U \coloneqq Q\bm V^2$ are illustrated in Figure \ref{fig:InvariantDomainDimension2}. Note that we chose the large maturity $T=100$ to give the process more time to explore its domain. Thereby, it is more clearly visible that the domain of $\bm U$ is indeed $\R_+^2$, than if we had set $T=1$.
\end{example}

\begin{figure}[h]
\centering
\begin{minipage}{.5\textwidth}
  \centering
  \includegraphics[width=\linewidth]{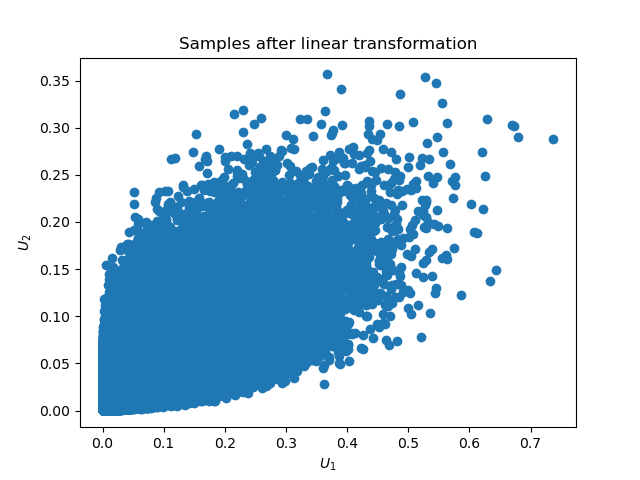}
\end{minipage}%
\begin{minipage}{.5\textwidth}
  \centering
  \includegraphics[width=\linewidth]{Samples_of_volatility_2_png.png}
\end{minipage}
\caption{Samples of $\bm V^2$ (right) and $\bm U$ (left) using $10^3$ sample paths on a time grid with $M=10^5$ time steps. The black lines correspond to the hyperplanes in \eqref{eqn:DomainEquationDimension2}. The parameters used are $\bm x = (1, 10), \bm w = (1, 2), \lambda=0.3, \nu=0.3, V_0=0.02, \theta=0.02, T=100,$ and $\bm v_0=V_0 / (2 \bm x) (w_1/x_1 + w_2/x_2),$ i.e. $\bm v_0$ is chosen to be proportional to $\bm x^{-1}$.}
\label{fig:InvariantDomainDimension2}
\end{figure}

For $N=3$, a similar computation -- relegated to the appendix due to its length -- gives multiple choices of domains. See Appendix~\ref{sec:InvariantDomainDimension3} for details.

\subsection{Proof of Theorem~\ref{thm:SpecificQ}}\label{S:proofnonempty}

In preparation for the proof, we introduce the matrix  $R = (r_{i,j})_{i,j=1, \ldots, N}$ defined by 
\begin{equation}\label{eq:R}
	\begin{aligned}
		r_{i,N} &= \frac{1}{\overline w},\qquad &&i = 1,\dots,N,\\
		r_{i,j} &= \frac{w_{j+1}}{\sum_{\ell=1}^j w_\ell \sum_{\ell=1}^{j+1} w_\ell},\qquad &&i \le j < N,\\
		r_{i+1,i} &= \frac{-1}{\sum_{\ell=1}^{i+1} w_\ell},\qquad &&i = 1,\dots,N-1,\\
		r_{i,j} &= 0,\qquad &&i \ge j + 2,
	\end{aligned}
\end{equation}
that will turn out to be the inverse of $Q$. 
For example, for $N=4$ we have  $$R = \begin{pmatrix}
	\frac{w_2}{w_1(w_1+w_2)} & \frac{w_3}{(w_1+w_2)(w_1+w_2+w_3)} & \frac{w_4}{(w_1+w_2+w_3)(w_1+w_2+w_3+w_4)} & \frac{1}{w_1+w_2+w_3+w_4}\\
	\frac{-1}{w_1+w_2} & \frac{w_3}{(w_1+w_2)(w_1+w_2+w_3)} & \frac{w_4}{(w_1+w_2+w_3)(w_1+w_2+w_3+w_4)} & \frac{1}{w_1+w_2+w_3+w_4}\\
	0 & \frac{-1}{w_1+w_2+w_3} & \frac{w_4}{(w_1+w_2+w_3)(w_1+w_2+w_3+w_4)} & \frac{1}{w_1+w_2+w_3+w_4}\\
	0 & 0 & \frac{-1}{w_1+w_2+w_3+w_4} & \frac{1}{w_1+w_2+w_3+w_4}
\end{pmatrix}$$
\begin{proof}[Proof of Theorem~\ref{thm:SpecificQ}]
	Conditions 2 and 3  of Definition \ref{def:AdmissibleMatrix} are readily satisfied by construction. To argue Condition 1, we will prove that $R$ given in \eqref{eq:R} is actually the inverse of $Q$, i.e.~$QR=\id$. Indeed, consider first the diagonal elements. Here, we have $$(QR)_{NN} = \sum_{k=1}^N q_{Nk}r_{kN} = \sum_{k=1}^N w_k \frac{1}{\overline w} = 1,$$ $$(QR)_{ii} = \sum_{k=1}^N q_{ik}r_{ki} = \sum_{k=1}^i w_k \frac{w_{i+1}}{\sum_{\ell=1}^i w_\ell \sum_{\ell=1}^{i+1} w_\ell} + \left(-\sum_{\ell=1}^i w_\ell\right) \frac{-1}{\sum_{\ell=1}^{i+1} w_\ell} = 1,$$ for $i=1,\dots,N-1$.
	Next, consider off-diagonal elements. We have $$(QR)_{Nj} = \sum_{k=1}^N q_{Nk}r_{kj} = \sum_{k=1}^j w_k \frac{w_{j+1}}{\sum_{\ell=1}^j w_\ell \sum_{\ell=1}^{j+1} w_\ell} + w_{j+1} \frac{-1}{\sum_{\ell=1}^{j+1} w_\ell} = 0,$$ for $j \le N-1$, $$(QR)_{iN} = \sum_{k=1}^N q_{ik}r_{kN} = \sum_{k=1}^i w_k \frac{1}{\overline w} + \left(-\sum_{\ell=1}^i w_\ell\right) \frac{1}{\overline w} = 0,$$ for $i\le N-1$, $$(QR)_{ij} = \sum_{k=1}^N q_{ik}r_{kj} = \sum_{k=1}^i w_k \frac{w_{j+1}}{\sum_{\ell=1}^jw_\ell\sum_{\ell=1}^{j+1} w_\ell} + \left(-\sum_{\ell=1}^i w_\ell\right)\frac{w_{j+1}}{\sum_{\ell=1}^jw_\ell\sum_{\ell=1}^{j+1} w_\ell} = 0,$$ for $i < j \le N-1$, and $$(QR)_{ij} = \sum_{k=1}^N q_{ik}r_{kj} = \sum_{k=1}^j w_k \frac{w_{j+1}}{\sum_{\ell=1}^jw_\ell\sum_{\ell=1}^{j+1} w_\ell} + w_{j+1}\frac{-1}{\sum_{\ell=1}^{j+1} w_\ell} = 0,$$ for $j < i \le N-1$. In particular, this proves that $R = Q^{-1}$.
	
	Finally, we verify  Condition 4 of Definition~\ref{def:AdmissibleMatrix} by direct computations. First,  it is easily verified that we have $Q\diag(\bm x) = S \coloneqq (s_{i,j})_{i,j=1,\ldots, N}$ with $$s_{ij} = q_{ij}x_j = w_jx_j,\qquad j \le i,\qquad s_{i,i+1} = -x_{i+1}\sum_{\ell=1}^i w_\ell,\qquad i = 1,\dots,N-1.$$ Now, let us compute $Q\diag(\bm x) Q^{-1} = T \coloneqq (t_{i,j})_{i,j=1,\ldots,N}.$
	
	 We have 
     $$
     t_{i,N} = \sum_{k=1}^N s_{i,k} r_{k,N} = \frac{1}{\overline{w}} \left(\sum_{k=1}^i w_kx_k - x_{i+1}\sum_{\ell=1}^i w_\ell\right) \le 0 
     $$ 
     for $i=1,\dots,N-1$, since the $x_i$ are ordered increasingly. Similarly, 
     $$
     t_{N,j} = \sum_{k=1}^N s_{N,k}r_{k,j} = \sum_{k=1}^j x_kw_k \frac{w_{j+1}}{\sum_{\ell=1}^j w_\ell \sum_{\ell=1}^{j+1} w_\ell} + x_{j+1}w_{j+1} \frac{-1}{\sum_{\ell=1}^{j+1}w_\ell} \le 0 
     $$
     for $j=1,\dots,N-1$. Next, 
     $$
     t_{i,j} = \sum_{k=1}^N s_{i,k} r_{k,j} = \left(\sum_{k=1}^i w_kx_k - x_{i+1}\sum_{\ell=1}^i w_\ell\right)\frac{w_{j+1}}{\sum_{\ell=1}^j w_\ell \sum_{\ell=1}^{j+1} w_\ell} \le 0 
     $$ 
     for $i < j \le N-1$. Finally, 
     $$
     t_{i,j} = \sum_{k=1}^N s_{i,k} r_{k,j} = \sum_{k=1}^j w_kx_k\frac{w_{j+1}}{\sum_{\ell=1}^j w_\ell \sum_{\ell=1}^{j+1} w_\ell} + w_{j+1}x_{j+1}\frac{-1}{\sum_{\ell=1}^{j+1} w_\ell} \le 0 
     $$ 
     for $j < i \le N-1$. This verifies Condition 4 of Definition \ref{def:AdmissibleMatrix} and proves the theorem.
\end{proof}

\subsection{Proof of Theorem~\ref{thm:SVEDomainTheorem}}\label{S:proofDomainInvariance}

Fix an admissible matrix $Q \in \mathcal Q$. The main idea of the proof is to reduce the study to the  process $\bm Z=Q \bm Y$ and prove that its associated SDE admits an $\mathbb R^N_+$-valued solution. 

We start by writing the SDE for $\bm Z$. For this we first observe that due to \eqref{eqn:InitialConditionsCondition}, $\bm Y$ satisfies 
$$\dd \bm Y_t = -\diag(\bm x)\bm Y_t \sdd t+ {b}_{\mu}(\bm w^\top \bm Y_t)\bm 1\sdd t + \sigma(\bm w^\top \bm Y_t)\bm 1\sdd W_t$$ 
where ${b}_{\mu}(z) = b(z) + \mu$. Using $Q$ as a transformation of basis (in the sense $\bm Z = Q\bm Y$), we get, thanks to the invertibility of $Q$, the following SDE
\begin{equation*}
	\dd \bm Z_t = -Q\diag(\bm x) Q^{-1}\bm Z_t\sdd t + b_{\mu}(\bm w^\top  Q^{-1}\bm Z_t)Q\bm 1\sdd t + \sigma(\bm w^\top Q^{-1}\bm Z_t)Q\bm 1 \sdd W_t.
\end{equation*}
Using the admissibility conditions 2 and 3 in  Definition~\ref{def:AdmissibleMatrix}, we have  that  $\bm w^\top Q^{-1} = \bm e_N^\top QQ^{-1} = \bm e_N^\top $ and $Q\bm 1 = \overline{w}\bm e_N$, which  simplifies the equation to 
\begin{align}
	\dd \bm Z_t &= -Q\diag(\bm x) Q^{-1}\bm Z_t\sdd t + \overline{w}b_{\mu}(Z^{(N)}_t)\bm e_N\sdd t + \overline{w}\sigma(Z^{(N)}_t)\bm e_N \sdd W_t.\label{eqn:SVEMarkovTransformedSimplified}
\end{align}
Recall that $Z^{(N)}$ is the $N$-th component of $\bm Z$. 

In order to prove Theorem~\ref{thm:SVEDomainTheorem}, it suffices to prove that for each $\bm Z_0 \in \R_+^N$, there exists an $\R_+^N$-valued $\bm Z$ weak solution to \eqref{eqn:SVEMarkovTransformedSimplified}. In particular, this would hold for any initial value of the form $\bm Z_0 = Q\bm Y_0$ with $Y_0 \in \mathcal D = Q^{-1}\R_N^+$  and setting $\bm Y = Q Z$, one obtains a $\mathcal D$-valued weak solution $\bm Y$ to  \eqref{eqn:SVEMarkovCompact} started at $\bm Y_0$. 

Hence, this boils down to establish that the set  $\R_+^N$ is stochastically viable with respect to the equation~\eqref{eqn:SVEMarkovTransformedSimplified}. Viability and invariance theory for stochastic differential equations have been extensively studied in the literature in various contexts and with different assumptions on the domain and the coefficients, we refer to  \citet*{abi2019stochastic, da2004invariance, da2007stochastic} and the references therein.

For the nonnegative orthant $\R^N_+$ the characterization in terms of the coefficients is very simple and means that, at boundary points, the diffusive coefficient has to be tangential to the boundary and the drift inward pointing. This is summarized in the following lemma. 

\begin{lemma}\label{L:invorthant}
	Let $\tilde b,\tilde \sigma:\R^N \to \R^N$ be continuous satisfying the growth conditions $$\|\tilde b(\bm z)\| + \|\tilde \sigma (\bm z)\| \le L(1 + \|\bm y\|),\qquad  \bm z\in\R^N,$$
	and the boundary conditions, for all $\bm z\in \R^N_+$, 
	\begin{align}\label{eq:conditionboundaryorthant}
		z_i= 0 \; \Rightarrow \;  \bm e_i^\top \tilde b(\bm z ) \geq 0 \text{ and } \bm e_i^\top \tilde \sigma(\bm z) = 0,  \quad i=1,\ldots, N, 
	\end{align}
then, for each $\widetilde{\bm Z}_0 \in \R_+^N$, there exists a weak $\R_+^N$-valued solution to the following SDE 
\begin{align*}
\sdd \widetilde{\bm Z}_t = \tilde b\left(\widetilde{\bm Z}_t\right) \sdd t + \tilde \sigma  \left(\widetilde{\bm Z}_t \right)\sdd W_t .
\end{align*}
\end{lemma} 

\begin{proof}
	See for instance \citet[Example 2.7]{da2007stochastic}.
\end{proof}

We now proceed to the proof of Theorem \ref{thm:SVEDomainTheorem}.
\begin{proof}[Proof of Theorem \ref{thm:SVEDomainTheorem}]
	It remains to apply Lemma~\ref{L:invorthant} on the equation \eqref{eqn:SVEMarkovTransformedSimplified}. For this, we define 
	$$ \tilde b(\bm z) =  -Q\diag(\bm x) Q^{-1}\bm z + \overline{w}b_{\mu}(z_N)\bm e_N  \quad \text{and} \quad \tilde \sigma (\bm z) = \overline{w}\sigma(z_N)\bm e_N , \quad z \in \R^N. $$
	Then, it readily follows from the continuity and growth conditions of $b$ and $\sigma$ that $\tilde b, \tilde \sigma$ are also continuous with at most linear growth conditions. As for the boundary conditions \eqref{eq:conditionboundaryorthant}, we fix $\bm z\in \R^N_+$ such that $z_i=0$ for some $i=1,\ldots, N$.
	
	 $\bullet$ For the diffusion term, we have 
	$$  \bm e_i^\top \tilde \sigma(\bm z)  = \overline{w}\sigma(z_N)  \bm  e_i^\top  \bm e_N  =0, $$
	since $\bm e_i^\top \bm e_N=0$ if $i<N$ and $\sigma(z_N)= \sigma(0)=0$ if $i=N$, where we used the boundary condition on $\sigma$ in \eqref{eqn:BoundaryCondition}. 

	$\bullet$ For the drift term, we first observe that for the same reason  $b_{\mu}(z_N) 	\bm e_i^\top\bm e_N= (b(z_N) + \mu ) \bm e_i^\top\bm e_N \geq 0$, since $b(0) + \mu \geq 0$ thanks to the boundary condition on $b$ in  \eqref{eqn:BoundaryCondition} and the fact that $\mu \geq 0$,  so that we can write  
	\begin{align*}
		\bm e_i^\top \tilde b(\bm z ) &=   - 	\bm e_i^\top  Q\diag(\bm x) Q^{-1}\bm z + \overline{w}b_{\mu}(z_N) 	\bm e_i^\top\bm e_N  \\
		&\geq  - \sum_{j \neq i} (Q\diag(\bm x) Q^{-1})_{ij} z_j \\
		&\geq 0,
	\end{align*}
 where  the first inequality follows from $z_i=0$ and the second inequality follows from the admissibility condition 4 in Definition~\ref{def:AdmissibleMatrix} for the matrix $Q$ and the fact that $z_j\geq 0$. 
 
 This shows that the boundary conditions \eqref{eq:conditionboundaryorthant} are satisfied by $\tilde b, \tilde \sigma$, so that an application of Lemma~\ref{L:invorthant} yields the existence of an $\R_+^N$-valued solution $\bm Z$ to  \eqref{eqn:SVEMarkovTransformedSimplified} for any initial condition $\bm Z_0\in \R_+^N$.  In particular, it holds for the initial value  $\bm Z_0 = Q\bm Y_0$ with $Y_0 \in \mathcal D = Q^{-1}\R^N_+$.  Setting $\bm Y = Q Z$, one obtains a $\mathcal D$-valued weak solution $\bm Y$ to  \eqref{eqn:SVEMarkovCompact} started at $\bm Y_0$ and completes the proof of theorem.  
	\end{proof}

\section{The weak scheme is cone-preserving}
\label{sec:WellDefinitenessOfWeakScheme}

In Section \ref{sec:StateSpaceMarkovianApproximation}, we determined the state space $\mathcal{D}\subseteq \R^N$ of the multifactor square-root process $\bm V^N$ given by \eqref{eqn:RHestonVolMarkovCompact}. Assume now that we approximate the process $\bm V^N$ using the weak simulation scheme proposed in \citet*{bayer2023++++Simulation}. The goal of this section is to prove that the resulting approximation has the same viable domain $\mathcal{D}$ as $\bm V^N$. 

Let us first start by recalling the weak simulation scheme of \citet*{bayer2023++++Simulation}. First, the SDE in \eqref{eqn:RHestonVolMarkovCompact} is split into two parts, one containing the drift and the other the diffusion. Denote by $D(\bm z, h) \coloneqq \bm Z_h \coloneqq (Z^{i}_h)_{i=1}^N$ the solution at time $h$ of the ordinary differential equation (ODE)
\begin{equation}\label{eqn:VMarkovDeterministicPart}
\dd Z^i_t = -x_i(Z^i_t - v^i_0)\sdd t + (\theta - \lambda Z^i_t) \sdd t,\quad Z^i_0 = z^i,\quad i=1,\dots,N,\quad Z_t = \bm{w}^\top \bm{Z}_t,
\end{equation}
and by $S(\bm{y}, h) \coloneqq \bm{Y}_h \coloneqq (Y^i_h)_{i=1}^N$ the solution at time $h$ of the SDE
\begin{equation}\label{eqn:VMarkovStochasticPart}
\dd Y^i_t = \nu\sqrt{Y_t}\sdd W_t,\qquad Y^i_0 = y^i,\qquad i=1,\dots, N,\qquad Y_t = \bm{w}^\top \bm{Y}_t.
\end{equation}

Then, the ODE \eqref{eqn:VMarkovDeterministicPart} is linear and can hence be solved exactly. Therefore, the simulation scheme $\widehat{D}$ for the ODE is simply given by
\begin{equation*}
    \widehat{D}(\bm z, h) \coloneqq D(\bm z, h) \coloneqq e^{Bh} \bm z + B^{-1} (e^{Bh} - \id)b, 
\end{equation*}
where
\begin{equation*}
B \coloneqq -\lambda \bm 1 \bm w^\top  - \diag(\bm x),\quad \textup{and} \quad b \coloneqq \theta \bm 1 + \diag(\bm x) \bm v_0.
\end{equation*}

We now recall the simulation scheme for the SDE \eqref{eqn:VMarkovStochasticPart}. Note that the right-hand side of \eqref{eqn:VMarkovStochasticPart} is the same for all $i$. Thus, after multiplying \eqref{eqn:VMarkovStochasticPart} with $\bm w$, we get $$dY_t = \nu\overline{w} \sqrt{Y_t} \sdd W_t,\qquad Y_0 = \bm w ^\top \bm y,$$ where $\overline{w} \coloneqq \bm 1^\top \bm w.$ This is now a one-dimensional SDE, which was already studied in \citet{lileika2021second}, where a second-order simulation scheme was given. This scheme is based on matching the first {6 centralized moments (up to errors of order $O(h^3)$),} while preserving the non-negativity of $Y$. Define the quantities
\begin{align}
x &\coloneqq \bm w^\top \bm y,\qquad z \coloneqq \nu^2\overline{w}^2 h,\label{eqn:DefineX}\\
m_1 &\coloneqq x,\qquad m_2 \coloneqq x^2 + xz,\qquad m_3 \coloneqq x^3 + 3x^2z + \frac{3}{2}xz^2,\nonumber\\
p_1 &\coloneqq \frac{m_1x_2x_3 - m_2(x_2 + x_3) + m_3}{x_1(x_3-x_1)(x_2-x_1)},\nonumber\\
p_2 &\coloneqq \frac{m_1x_1x_3 - m_2(x_1 + x_3) + m_3}{x_2(x_3-x_2)(x_1-x_2)},\nonumber\\
p_3 &\coloneqq \frac{m_1x_1x_2 - m_2(x_1 + x_2) + m_3}{x_3(x_1-x_3)(x_2-x_3)},\nonumber\\
x_1 &\coloneqq x + \left(a + \frac{3}{4}\right)z - \sqrt{\left(3x + \left(a + \frac{3}{4}\right)^2z\right) z},\label{eqn:X1}\\
x_2 &\coloneqq x + a z,\label{eqn:X2}\\
x_3 &\coloneqq x + \left(a + \frac{3}{4}\right)z + \sqrt{\left(3x + \left(a + \frac{3}{4}\right)^2z\right) z},\label{eqn:X3}\\
a &\coloneqq \frac{3 + \sqrt{3}}{4}.\nonumber
\end{align}
Then, we define $\widehat Y_h$ to be the random variable which is $x_i$ with probability $p_i$, $i=1,2,3$, {noting, in particular, that $p_1+p_2+p_3 = 1$}.

We can now reconstruct an approximation $\widehat{\bm{Y}}$ from $\widehat Y$. Indeed, since the right-hand side of \eqref{eqn:VMarkovStochasticPart} is the same for all $i=1,\dots,N$, the solution of \eqref{eqn:VMarkovStochasticPart} must be of the form
\begin{equation}\label{eqn:BYForm}
Y^i_h = y^i + R,\qquad i = 1,\dots,N,
\end{equation}
for some scalar random variable $R$. Taking the inner product of \eqref{eqn:BYForm} with $\bm w$, we get $$Y_h = \bm w^\top \bm y + \overline{w}R,\quad \textup{implying}\quad R = \frac{Y_h - \bm w^\top \bm y}{\overline{w}}.$$ Hence, we set $$\widehat{S}(\bm y, h) \coloneqq \widehat{\bm Y}_h \coloneqq \bm y + \frac{\widehat Y_h - \bm w^\top \bm y}{\overline{w}}.$$

Finally, we use Strang splitting to get the scheme $$A^{\textup{CIR}}(\bm{v},h) \coloneqq D\left(\widehat{S}\left(D\left(\bm{v}, \frac{h}{2}\right), h\right), \frac{h}{2}\right)$$ for approximating $\bm{V}_h$ given $\bm{v}$. Therefore, we get a simulation algorithm $$\bm V^{N, M}_{t_{j+1}} \coloneqq A^{\textup{CIR}}(\bm V^{N, M}_{t_j}, t_{j+1} - t_j),\qquad j=0,\dots,M-1,$$ where $0 = t_0 < t_1 < \dots < t_M = T.$ The only problem that could occur is that the square root in \eqref{eqn:X1} or \eqref{eqn:X3} is not well-defined. However, note that if we can prove that $\bm V^{N,M}$ does not leave $\mathcal{D}$, where $\mathcal{D}$ is the same cone as in Theorem \ref{thm:SVEDomainTheorem}, then in particular, $x = \bm w^\top \bm y$ in \eqref{eqn:DefineX} will always be nonnegative, and hence the square roots in \eqref{eqn:X1} and \eqref{eqn:X3} are always well-defined. Proving that $\bm V^{N,M}$ stays in $\mathcal{D}$ is the aim of the following theorem.

\begin{theorem}
    Let $Q$ be an admissible matrix and let $\bm v_0$ be chosen according to \eqref{eqn:InitialConditionsCondition}. Then, for all $\bm v\in Q^{-1}\R_+^N$ and $h\ge 0$, the weak simulation algorithm $A^{\textup{CIR}}$ described above satisfies $A^{\textup{CIR}}(\bm v, h) \in Q^{-1}\R_+^N$. In particular, $A^{\textup{CIR}}$ is well-defined.
\end{theorem}

\begin{proof}
    Given $\bm z,\bm y\in Q^{-1}\R_+^N$ and $h\ge 0$, we want to show that $D(\bm z, h)\in Q^{-1}\R_+^N$, and $\widehat S(\bm y, h)\in Q^{-1}\R_+^N$. This will prove the theorem.

    Consider first the algorithm $\widehat S$. Recall that $\widehat S(\bm y, h) = \bm y + R \bm 1$ for some scalar random variable $R$. We have to verify that $Q\widehat S(\bm y, h) = Q\bm y + RQ\bm 1\in\R_+^N$. The last component of this vector is given by $$(Q\widehat S(\bm y, h))_N = \bm w^\top  \bm y + R\overline w,$$ and we recall that this was given by the random variable $\widehat Y_h$ in Section \ref{sec:WellDefinitenessOfWeakScheme}, which by definition is nonnegative, as verified in \citet*{lileika2021second}. Conversely, for $i=1,\dots,N-1$, we have $$(Q\widehat S(\bm y, h))_i = (Q\bm y)_i + 0 \ge 0$$ by the assumption that $Q\bm y \in\R_N^+$. Hence, $\widehat S$ leaves the domain $Q^{-1}\R_+^N$ invariant.

    Next, consider the algorithm $D$. Recall that $D(\bm z, h)$ was given as the exact solution at time $h$ of the ODE $$\dd \bm Z_t = -\diag(\bm x) (\bm Z_t - \bm v_0)\sdd t + (\theta - \lambda \bm w^\top  \bm Z_t)\bm 1 \sdd t,\qquad \bm Z_0 = \bm z.$$ Note that due to \eqref{eqn:InitialConditionsCondition}, $\diag(\bm x) \bm v_0 = \mu\bm 1$ for some $\mu\ge 0$. Defining $\widetilde{\bm Z} \coloneqq Q\bm Z$, we have $$\dd \widetilde{\bm Z}_t = \left(-Q\diag(\bm x) Q^{-1}\widetilde{\bm Z_t} + (\theta + \mu - \lambda \widetilde{Z}^N_t)Q\bm 1 \right)\dd t,\qquad \widetilde{\bm Z}_0 = Q\bm z\in \R_+^N,$$ and we have to show that $\widetilde{\bm Z}_t \in \R_+^N$. We prove this by invoking Lemma~\ref{L:invorthant}, where we note that $$b(\bm z) = -Q\diag(\bm x) Q^{-1}\bm z + (\theta + \mu - \lambda z^N)Q\bm 1,\qquad \sigma\equiv 0.$$ In particular, we have to show that $b_i(\bm z) \ge 0$ for $\bm z\in\R_+^N$ with $z^i = 0$.

    We start with $i=N$. Here, we have
    \begin{align*}
        b_N(\bm z) &= -(Q\diag(\bm x) Q^{-1}\bm z)_N + (\theta + \mu)\overline w.
    \end{align*}
    Of course, $(\theta + \mu)\overline w \ge 0$. Moreover, due to Definition \ref{def:AdmissibleMatrix}, $-(Q\diag(\bm x)Q^{-1}\bm z)_N$ is a linear combination of $z^i$ where all the coefficients are nonnegative, with the exception of the coefficient of $z^N$. However, since $z^N = 0$, this implies that $b_N(\bm z) \ge 0$.

    Next, consider $i=1,\dots,N-1$. Then, $$b_i(\bm z) = -(Q\diag(\bm x)Q^{-1}\bm z)_i.$$ As before, $-(Q\diag(\bm x)Q^{-1}\bm z)_i$ is again a linear combination of the $z^j$, where all coefficients are nonnegative, with the exception of the coefficient of $z^i$. But since $z^i = 0$, this implies that $b_i(\bm z) \ge 0$, proving the theorem.
\end{proof}

\section{Solving PDEs}
\label{sec:solving_pdes}

As an application of the domain, we want to solve PDEs. Recall that the multifactor square-root process $\bm V^N$ is given by 
$$
\dd \bm V^N_t = -\diag(\bm x)\left(\bm V^N_t - \bm v_0\right)\sdd t + \left(\theta-\lambda\bm w^\top \bm V^N_t\right)\bm 1\sdd t + \nu\sqrt{\bm w^\top \bm V^N_t}\bm 1\sdd W_t,
$$ 
see \eqref{eqn:RHestonVolMarkovCompact}.
After a transformation of variables using $\bm Z \coloneqq Q\bm V^N$ and $\bm z_0 \coloneqq Q\bm v_0$, where $Q$ is the matrix in Theorem \ref{thm:SpecificQ}, we have $$\dd \bm Z_t = -Q\diag(\bm x)Q^{-1}\left(\bm Z_t - \bm z_0\right)\sdd t + \overline{w}\left(\theta-\lambda Z^{(N)}_t\right)\bm e_N\sdd t + \nu \overline{w}\sqrt{Z^{(N)}_t}\bm e_N\sdd W_t.$$

Let $f:\R^N_+\to\R$ be a ``nice'' payoff function. Then, we define the value function $u:\R^N_+\times[0,T]\to\R$, $$u(\bm z, t) \coloneqq \E\left[f(\bm Z_T)\big|\bm Z_t = \bm z\right].$$ Then, $u$ satisfies the PDE 
\begin{equation}
    \label{eq:PDE}
    \partial_t u - (\nabla u)^\top  Q\diag(\bm x)Q^{-1}(\bm z-\bm z_0) + \overline{w}\left(\theta-\lambda z_N\right)\partial_{z_N}u + \frac{1}{2}\nu^2\overline{w}^2z_N\partial^2_{z_N}u = 0
\end{equation}
with the boundary condition $u(\bm z, T) = f(\bm z)$, $\bm z\in\R_+^N$.

For numerical approximation, we then need to truncate the domain in space, and impose appropriate boundary conditions.
For simplicity, we will instead fabricate an appropriate source term such that the PDE has an explicit, given solution, which we then also impose as Dirichlet boundary condition on the boundary of the truncated domain.

Specifically, suppose that we want the exact solution to have the form
\[
    u(\mathbf{z}, t) = \widetilde{u}(\mathbf{z}, t) := 1 + \sum_{i=1}^N \alpha_i (z^i)^2 + \beta t, \quad \mathbf{z} \in \R^N_+,\ t \in [0,T].
\]
Plugging this formula into \eqref{eq:PDE}, we obtain a source term
\[
    \phi(\mathbf{z}) = \beta - 2 \sum_{i=1}^N \alpha_i z^i \sum_{j=1}^N g_{ij} (z^j - z_0^j) + 2 \alpha_N \overline{w} (\theta - \lambda z^N) z^N
    + \nu^2 \overline{w}^2 \alpha_N z^N,
\]
with $g_{ij} = (Q\diag(\bm x)Q^{-1})_{ij}$, i.e., $u$ satisfies
\[
    \partial_t u - (\nabla u)^\top  Q\diag(\bm x)Q^{-1}(\bm z-\bm z_0) + \overline{w}\left(\theta-\lambda z_N\right)\partial_{z_N}u + \frac{1}{2}\nu^2\overline{w}^2z_N\partial^2_{z_N}u = \phi,
\]
now with the terminal condition $u(\mathbf{z}, T) = \widetilde{u}(\mathbf{z},T)$.
The precise parameters chosen are summarized in Table~\ref{tab:rheston_params}. In dimension $N=2$, we choose the admissible matrix $Q$ given by Example~\ref{ex:InvariantDomainDimension2}.
We furthermore choose $\mathbf{\alpha} = (3,\, 4)$, $\beta = 1.6$, and the terminal time $T=2$.

\begin{table}[]
    \centering
    \begin{tabular}{ccccccc}
         $N$ & $\theta$ & $\lambda$ & $\nu$ & $\mathbf{x}$ & $\mathbf{w}$ & $\mathbf{v}_0$ \\
         \hline
          $2$ & $0.8$ & $1.2$ & $0.7$ & $(0.1,\, 3.5)$ & $(0.4,\, 1.8)$ & $(0.2,\, 0.3)$ \\
          $3$ & $0.8$ & $1.2$ & $0.7$ & $(0.1,\, 3.5,\, 4.1)$ & $(0.4,\, 1.8,\, 2.1)$ & $(0.2,\, 0.3,\, 0.4)$
    \end{tabular}
    \caption{Parameters of the stochastic volatility of the lifted rough Heston model used for the numerical example.}
    \label{tab:rheston_params}
\end{table}

After truncation of the domain, we solve the PDE by the finite element method, using the package FEniCSx, see \citet{BarattaEtal2023}, and compare against the exact solution $\widetilde{u}$.
In Table~\ref{tab:fem-error} we present the $L^2$-errors on the truncated domain for three choices of truncated domains, each of side-length $4$:
\begin{enumerate}
    \item \label{enum:1} $D = [0,4]^2$, corresponding to a truncation in $v$-space which respects the cone-shaped actual domain of the process;
    \item \label{enum:2} $D = [-0.5,3.5]^2$ corresponding to a truncation in $v$-space, which neither respects the cone-shaped actual domain, nor the non-negativity condition;
    \item \label{enum:3} $D = [-0.5,3.5] \times [0,4]$ corresponding to a domain truncation, which does not respect the cone-shaped actual domain in $v$-space, but does respect the non-negativity.
\end{enumerate}
 We use first order Lagrange-type finite elements, with $n_t$ time-steps as well as mesh-size $n_x = n_t$ in each space dimension. (We refer to \url{https://github.com/bayerc2/domain_multifactor_volterra} for more details.)

 We would like to emphasize that, while it might seem trivial to choose $[0,4]^2$ as the domain instead of, say, $[-0.5,3.5] \times [0,4]$, this choice is based on correctly identifying the matrix $Q$ and the domain $\mathcal{D} = Q^{-1}\mathbb{R}_+^2$, which is appropriately truncated here to $Q^{-1}[0,4]^2$. Without knowledge of $Q$, one would need to guess $\mathcal{D}$ to truncate the domain, a task that becomes increasingly nontrivial in higher dimensions.

{\renewcommand{\arraystretch}{1.2}
\begin{table}[]
    \centering
    \begin{tabular}{r|c|c|c}
        & \multicolumn{3}{|c}{$L^2$-error over the domain $D$}\\
         $n$ & $D = [0,4]^2$ & $D = [-0.5, 3.5]^2$ & $D = [-0.5,3.5] \times [0,4]$ \\
         \hline 
         $4$ & $7.3 \times 10^{1}$ & $3.0 \times 10^{3}$ & $5.5 \times 10^{1}$ \\
         $8$ & $1.4 \times 10^{1}$ & $7.7 \times 10^{1}$ & $1.5  \times 10^{1}$ \\
         $16$ & $3.2 \times 10^{0}$ & $2.2 \times 10^{10}$ & $3.3 \times 10^{0}$ \\
         $32$ & $7.5 \times 10^{-1}$ & $1.5 \times 10^{80}$ & $8.0 \times 10^{-1}$ \\
         $64$ & $1.8 \times 10^{-1}$ & $1.4 \times 10^{50}$ & $2.0 \times 10^{-1}$ \\
         $128$ & $4.6 \times 10^{-2}$ & inf & $4.9 \times 10^{-2}$ \\
         $256$ & $1.1 \times 10^{-2}$ & inf & $1.2 \times 10^{-2}$ \\
         $512$ & $2.9 \times 10^{-3}$ & inf & $3.0 \times 10^{-3}$ \\
         $1024$ & $7.2 \times 10^{-4}$ & inf & $7.6 \times 10^{-4}$ \\
    \end{tabular}
    \caption{$L^2$ errors over the truncated domain for the approximate FEM solution to the PDE for $n_t = n_x = n$ in dimension $N=2$.}
    \label{tab:fem-error}
\end{table}
}

When non-negativity of the variance process is preserved (cases \ref{enum:1} and \ref{enum:3}), the numerical method empirically exhibits second order convergence, with slightly smaller error when the computational domain is a subset of the invariant domain of the process (case \ref{enum:1}).
On the other hand, when non-negativity of the variance process is not preserved on the computational domain (case \ref{enum:2}) the error explodes due to the instability of the heat equation backward in time.

{We also provide an example in dimension $N = 3$. In this case, we follow the construction outlined in Appendix~\ref{sec:InvariantDomainDimension3}. Note that the construction is not unique, so we tested different solutions to the system of inequalities leading to different admissible $Q$. Specifically, we choose
\begin{itemize}
    \item $a=1$, $b=2$, the default choice suggested in Section~\ref{sec:InvariantDomainDimension3}.
    \item $a = 2.04$, $b=0.88$, an admissible choice obtained by minimizing $\left\lVert Q \diag(\mathbf{x}) Q^{-1} \right\rVert$ -- related to the Lipschitz constant of the transformed drift -- over all admissible choices of parameters $a,b$.
    \item $a = 0.51$, $b = -0.03$, obtained by maximizing $\left\lVert Q \diag(\mathbf{x}) Q^{-1} \right\rVert$.
\end{itemize}
It turns out that there was no significant difference in the numerical results based on the different choices of transformation. Hence, we only report the results for the first choice. We report our numerical results in Table~\ref{tab:fem-error3}. Note that we had to limit the grid-sizes due to the severely increased computational time. Hence, the accuracies reported may seem disappointing. However, note that the $L^2$-error is un-normalized here. A normalized error would be obtained by dividing by the total volume of the (computational) domain, which is $4^3 = 64$ in this case. We observe an expected error decay when the domain is respected, and highly erratic, diverging behavior when positivity is not preserved.
}

{{\renewcommand{\arraystretch}{1.2}
\begin{table}[]
    \centering
    \begin{tabular}{r|c|c}
        & \multicolumn{2}{|c}{$L^2$-error over the domain $D$}\\
         $n$ & $D = [0,4]^3$ & $D = [-0.5, 3.5]^3$ 
         \\
         \hline 
         $4$ & $4.3 \times 10^{3}$ & $2.2 \times 10^{1}$ 
         \\
         $8$ & $3.2 \times 10^{2}$ & $1.3 \times 10^{3}$ 
         \\
         $16$ & $4.8 \times 10^{1}$ & $2.6 \times 10^{13}$ 
         \\
         $32$ & $8.5 \times 10^{0}$ & $1.6 \times 10^{24}$ 
         \\
         $64$ & $1.8 \times 10^{0}$ & $3.1 \times 10^{2}$ 
         \\
         $128$ & $4.0 \times 10^{-1}$ & $3.1 \times 10^{2}$ 
         \\
    \end{tabular}
    \caption{$L^2$ errors over the truncated domain for the approximate FEM solution to the PDE for $n_t = n_x = n$ in dimension $N=3$.}
    \label{tab:fem-error3}
\end{table}
}
}

\section{Remark on the link between the sets $\mathcal E$ and $\mathcal G$}
\label{A:link}
In this section, we argue that the two abstract `invariance' sets that appeared in the literature in \citet*{abi2019markovian, cuchiero2020generalized} are equal.  This part is valid for more general locally square-integrable kernels  $K$ beyond the weighted sum of exponential case.

We introduce the following notations. For suitable functions $f,g$ and measure $L$ we denote their convolution by $*$: 
$$ (f*g)(t) = \int_0^t f(t-s)g(s) \dd s = \int_0^t f(s)g(t-s) \dd s , \quad (f*L)(t):= \int_0^t f(t-s) L(\dd s). $$
The shift operator $\Delta_h$ with $h\ge0$, maps any function $f$ on $\R_+$ to the function $\Delta_h f$ given by
\[
\Delta_h f(t) = f(t+h).
\]
If the function $f$ on $\R_+$ is right-continuous and of locally bounded variation, the measure induced by its distributional derivative is denoted $\dd f$, so that $f(t) = f(0) + \int_{[0,t]} \dd f(s)$ for all $t\ge0$. By convention, $df$ does not charge $\{0\}$.\\

 Two sets appeared so far in the literature to characterize the non-negativity of solutions to stochastic Volterra equations:
\begin{enumerate}
    \item Set of     \citet[Equation (4.7), Definition 4.12 and Theorem 4.17(i)]{cuchiero2020generalized}:\footnote{We point out that in \cite[Equation~(4.7)]{cuchiero2020generalized}, 
 the invariance condition is expressed on the initial values of the 
Markovian lift~$\lambda_0$ of the Volterra process. In our setting, we 
formulate the invariance property directly in terms of the Volterra 
process itself, with input curve~$g_0$. The connection between the two sets is made when one considers input curves of the form
$g_0(t) = \langle g, \mathcal{S}_t^* \lambda_0 \rangle,$ with the notations of  \cite[Equation~(4.7)]{cuchiero2020generalized}.}
\begin{align*}
    \mathcal E = \bigcap_{\eta>0} \mathcal E^{\eta} \text{ with } \mathcal E^{\eta}:=  \left\{ g_0:[0,T]\to \R \text{ such that }  g_0 - R^{\eta}*g_0 \geq 0 \right\}.
\end{align*}
Here $R^{\eta}(t)$ is the resolvent of the second kind of the kernel $(\eta K)$ defined by 
\begin{align*}
    R^{\eta} = \eta K - \eta K*R^{\eta}= \eta K- R^{\eta}*\eta K.
\end{align*}
\item 
Set of \citet[Equations (2.4)-(2.5) and Theorem 2.1]{abi2019markovian}: 
\begin{align*}
    \mathcal G = \big\{  g_0:&[0,T]\to \R \text{ such that } \\
    &\Delta_h g_0 - (\Delta_h K*L)(0) g_0 -  d(\Delta_h K*L)*g_0 \geq 0   \text{ and } g_0(0) \geq 0. \big\}
\end{align*}
  where $L(\dd t)$ is the resolvent of the first kind of the kernel\footnote{Under some suitable assumptions on the kernel, see \cite[Assumption (H1)]{abi2019markovian},  one can show that $K$ admits a resolvent of the first kind such that  $\Delta_h K * L$ is right-continuous and of locally bounded variation, see \cite[Remark B.3]{abi2019markovian}, thus the
associated measure $d(\Delta_h K * L)$ that appears in the set $\mathcal G$ is well defined.} 
  \begin{align*}
    K*L = 1 = L*K.  
  \end{align*}
\end{enumerate}
  One can argue that the two sets are equal: 
  $$ \mathcal E = \mathcal G $$
since both conditions that appear in the set are necessary and sufficient conditions for  the non-negativity of the  linear Volterra equation 
\begin{align}\label{eq:Volterralinear}
    f^{\eta} = g_0 - \eta K*f^{\eta}, 
\end{align}
for $\eta>0$. Indeed, on the one hand the solution of \eqref{eq:Volterralinear} can be expressed in terms of the resolvent of the second kind in the form 
  $$ f^{\eta} = g_0 - R^{\eta}*g_0,$$
  which is exactly the form that appears in $\mathcal E$. On the other hand, by relying on the properties of the resolvent of the first kind, see for instance \citet[the proof of Theorem A.2]{abi2019markovian}, one can write that
  \begin{equation}\label{eq:feta}
      \begin{aligned}
    f^{\eta}(t+h) &=  \Delta_h g_0(t) - (\Delta_h K*L)(0) g_0(t) -  (d(\Delta_h K*L)*g_0)(t)  \\
    &\quad   + (\Delta_h K*L)(0) f^{\eta} (t) +  (d(\Delta_h K*L)*f^{\eta})(t)  \\
    &\quad   - \eta \int_t^{t+h} \Delta_h K(t-s) f^{\eta} (s) \dd s.
\end{aligned}
  \end{equation}
We note that the first line is exactly the condition that appears in the set $\mathcal G$. 
{We sketch why $g_0 \in \mathcal{G}$ is equivalent to the nonnegativity of $f^{\eta}$ for all $\eta > 0$, under suitable assumptions on the nonnegative kernel $K$ (for instance, when $K$ is a weighted sum of exponentials; see \cite[$(H_0)$ and $(H_1)$]{abi2019markovian} for precise conditions). The implication $\Rightarrow$ follows from \cite[Theorem~A.2]{abi2019markovian} with $b(x) = -\eta x$ and $\sigma(x) = 0$ therein.   For the converse direction, fix $\eta > 0$ and assume that $f^{\eta} \geq 0$. Clearly, $g_0(0) = f^{\eta}(0) \geq 0$, and from~\eqref{eq:feta} we have
\begin{align*}
\Delta_h g_0(t) 
&- (\Delta_h K * L)(0)\, g_0(t) 
- (d(\Delta_h K * L) * g_0)(t) \\
&\geq 
- (\Delta_h K * L)(0)\, f^{\eta}(t) 
- (d(\Delta_h K * L) * f^{\eta})(t) \\
&\quad 
+ \eta \int_t^{t+h} \Delta_h K(t-s)\, f^{\eta}(s)\, \dd s \\
&\geq 
- (\Delta_h K * L)(0)\, f^{\eta}(t) 
- (d(\Delta_h K * L) * f^{\eta})(t).
\end{align*}
To conclude that $g_0 \in \mathcal{G}$, it suffices to note that the right-hand side tends to zero as $\eta \to \infty$, since $f^{\eta} \to 0$ in this limit (this follows from the Laplace transform $\widehat{f^{\eta}}$ of~\eqref{eq:Volterralinear}, given by
$\widehat{f^{\eta}} = \frac{\widehat{g_0}}{1 + \eta \widehat{K}}$, which goes  to $0$ as $\eta \to \infty$). 
}

In conclusion,  the non-negativity of the linear Volterra equation \eqref{eq:Volterralinear}, for any $\eta>0$, is equivalent to the condition that appears in $\mathcal E$ as well as the one that appears in $\mathcal G$, which shows that the two sets are equal. \\

In principle, to establish a link with our cone $\mathcal D$, one should restrict to kernels that are weighted sums of exponentials of the form 
$$ 
K(t) = \sum_{i=1}^N w_i e^{-x_i t},
$$ 
and input curves of the form 
$$ 
g_0(t) = \sum_{i=1}^N w_i e^{-x_i t} Y_0^i.
$$

Then, the resolvents of the second kind and first kind for such kernels must be computed and plugged into the conditions defining the sets $\mathcal{E}$ and $\mathcal{G}$. Even in dimension $N=2$, this leads to highly cumbersome and non-trivial computations, and it is not clear how to explicitly determine a suitable domain, as for instance our cone $\mathcal D$, for $Y_0^i$ from $\mathcal{E}$ and $\mathcal{G}$. This makes the approach in the current paper particularly crucial.

\appendix

\section{Invariant domains in dimension $N = 3$}
\label{sec:InvariantDomainDimension3}
We extend the calculations presented for $N=2$ in Example~\ref{ex:InvariantDomainDimension2} to the three-dimensional case. We are looking for a matrix of the form
    \begin{equation}\label{eqn:QDimension3}
    Q = \begin{pmatrix}
        a_1 & a_2 & -a_1-a_2\\
        b_1 & b_2 & -b_1-b_2\\
        w_1 & w_2 & w_3
    \end{pmatrix}.
    \end{equation}
    Note that there are some scaling invariances in the equation $Q\bm x\in\R_+^N$. We may multiply rows of $Q$ with positive (!) constants without changing this condition. Hence, we restrict ourselves to $$Q = \begin{pmatrix}
        1 & -a & -1+a\\
        1 & b & -1-b\\
        w_1 & w_2 & w_3
    \end{pmatrix}.$$
    Note that this corresponds to the assumption that $a_1$ and $b_1$ in \eqref{eqn:QDimension3} are both positive. Indeed, if we chose one of these entries to be $0$ or $-1$, we would fail to find an appropriate matrix $Q$.
    
    Define the matrix $R\coloneqq -\overline w Q\diag(\bm x)Q^{-1}$, where we denote $R = (r_{i,j})_{i,j=1}^N$, and set $y_1\coloneqq x_2-x_1$ and $y_2 \coloneqq x_3-x_2$. Then,
    \begin{align*}
        r_{1,3} &= y_1 + y_2 - ay_2,\\
        r_{2,3} &= y_1 + y_2 + by_2,\\
        r_{3,2} &= \frac{(w_1w_2y_1 + w_1w_3(y_1+y_2))a - w_1w_2y_1 + w_2w_3y_2}{a+b},\\
        r_{3,1} &= \frac{(w_1w_2y_1 + w_1w_3(y_1+y_2))b + w_1w_2y_1 - w_2w_3y_2}{a+b},\\
        r_{1,2} &= \frac{w_1y_2a^2 + (w_3y_1 + w_2(y_1+y_2) - w_1y_2)a - w_2(y_1+y_2)}{a+b},\\
        r_{2,1} &= \frac{-w_1y_2b^2 + (w_3y_1 + w_2(y_1+y_2) - w_1y_2)b + w_2(y_1+y_2)}{a+b},
    \end{align*}
    and all these quantities have to be non-negative. Assume now further that $a,b\ge 0$. Then, $r_{2,3} \ge 0$ is trivially satisfied, and $r_{1,3},r_{3,2},r_{3,1},r_{1,2},r_{2,1} \ge 0$ simplify to 
    \begin{align*}
        a &\le \frac{y_1 + y_2}{y_2},\\
        a &\ge \frac{w_2}{w_1}\frac{w_1y_1 - w_3y_2}{w_2y_1 + w_3(y_1+y_2)},\\
        b &\ge \frac{w_2}{w_1}\frac{-w_1y_1 + w_3y_2}{w_2y_1 + w_3(y_1+y_2)},\\
        0 &\le w_1y_2a^2 + ca - w_2(y_1+y_2),\\
        0 &\ge w_1y_2b^2 - cb - w_2(y_1+y_2),
    \end{align*}
    where $c \coloneqq w_3y_1 + w_2(y_1+y_2) - w_1y_2.$

    This further implies for $a$ that
    \begin{align*}
        \frac{-c + \sqrt{c^2 + 4w_1w_2y_2(y_1+y_2)}}{2w_1y_2} \lor \frac{w_2}{w_1}\frac{w_1y_1 - w_3y_2}{w_2y_1 + w_3(y_1+y_2)} \le a \le \frac{y_1 + y_2}{y_2}.
    \end{align*}
    One can verify that the lower bound is always smaller than the upper bound, proving that such an $a$ exists. However, it is slightly simpler and perhaps more illustrative to prove that $a = 1$ satisfies these inequalities. For the upper bound, this is trivial. For the lower bound, note that $$\frac{w_2}{w_1}\frac{w_1y_1 - w_3y_2}{w_2y_1 + w_3(y_1+y_2)} \le \frac{w_2}{w_1}\frac{w_1y_1}{w_2y_1} = 1,$$ and 
    \begin{align*}
        \frac{-c + \sqrt{c^2 + 4w_1w_2y_2(y_1+y_2)}}{2w_1y_2} &\le 1\\
        \Longleftrightarrow \sqrt{c^2 + 4w_1w_2y_2(y_1+y_2)} &\le 2w_1y_2 + c\\
        \Longleftrightarrow c^2 + 4w_1w_2y_2(y_1+y_2) &\le c^2 + 4w_1^2y_2^2 + 4w_1y_2c\\
        \Longleftrightarrow w_2(y_1+y_2) &\le w_1y_2 + c.
    \end{align*}
    which follows immediately from the definition of $c$.

    Next, for $b$ we get the conditions
    \begin{align*}
        0 \lor \frac{w_2}{w_1}\frac{-w_1y_1 + w_3y_2}{w_2y_1 + w_3(y_1+y_2)} \le b \le \frac{c + \sqrt{c^2 + 4w_1w_2y_2(y_1+y_2)}}{2w_1y_2}.
    \end{align*}
    This time, we verify that $b = \frac{w_2}{w_1}$ is admissible. For the lower bound, this is clear. For the upper bound, note that
    \begin{align*}
        \frac{w_2}{w_1} &\le \frac{c + \sqrt{c^2 + 4w_1w_2y_2(y_1+y_2)}}{2w_1y_2}\\
        \Longleftrightarrow 2w_2y_2 - c &\le \sqrt{c^2 + 4w_1w_2y_2(y_1+y_2)}\\
        \Longleftarrow c^2 + 4w_2^2y_2^2 - 4w_2y_2c &\le c^2 + 4w_1w_2y_2(y_1+y_2)\\
        \Longleftrightarrow w_2y_2 - c &\le w_1(y_1+y_2).
    \end{align*}
    This again follows from the definition of $c$.

    Hence, we have shown that we can choose $a = 1$ and $b = \frac{w_2}{w_1}$, yielding $$Q = \begin{pmatrix}
        1 & -1 & 0\\
        1 & \frac{w_2}{w_1} & -1-\frac{w_2}{w_1}\\
        w_1 & w_2 & w_3
    \end{pmatrix}.$$ Note that due to scaling invariance, the matrix $$Q = \begin{pmatrix}
        w_1 & -w_1 & 0\\
        w_1 & w_2 & -w_1-w_2\\
        w_1 & w_2 & w_3\\
    \end{pmatrix}$$ would be equivalent.

    Consider now the specific example $\bm x \coloneqq (1, 5, 25)$ and $\bm w \coloneqq (1, 2, 3).$ Then, we get the conditions
    \begin{align*}
        0.84 \approx \frac{-5 + \sqrt{85}}{5} &\le a \le \frac{6}{5} = 1.2,\\
        1.4 = \frac{7}{5} &\le b \le \frac{5 + \sqrt{85}}{5} \approx 2.84.
    \end{align*}
    Comparing to the previous discussion, we see that indeed, $a = 1$ and $b=\frac{w_2}{w_1} = 2$ are admissible.

The corresponding plots for the multifactor square-root process are shown in Figure \ref{fig:InvariantDomainDimension3}. We give projections to two-dimensional planes, as this makes it easier to visually verify that the samples lie in $\R_+^3$. Furthermore, we give three different choices of $(a, b)$, the first two being admissible, and the third not. Indeed, we see for the first two choices that the samples lie in $\R_+^3$, while this is not the case for the third choice.

\begin{figure}[h!]
\centering
\begin{minipage}{.33\textwidth}
  \centering
  \includegraphics[width=\linewidth]{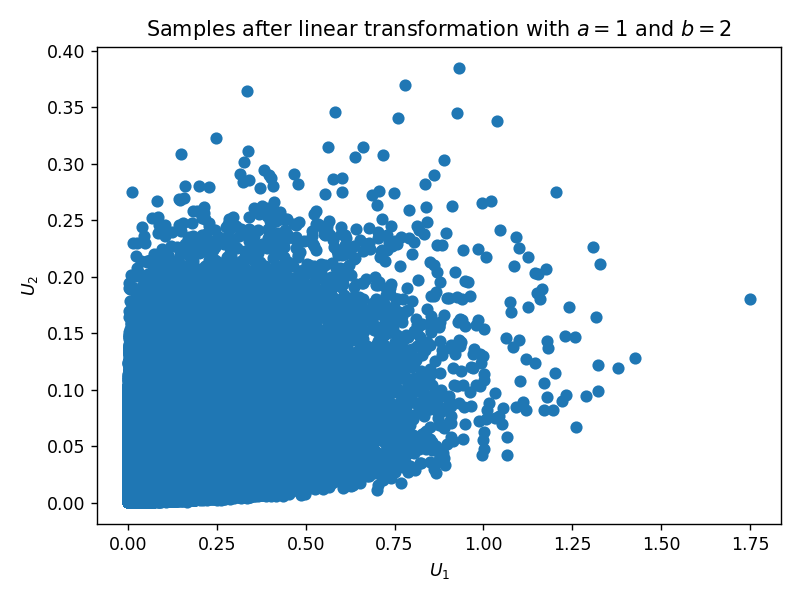}
\end{minipage}%
\begin{minipage}{.33\textwidth}
  \centering
  \includegraphics[width=\linewidth]{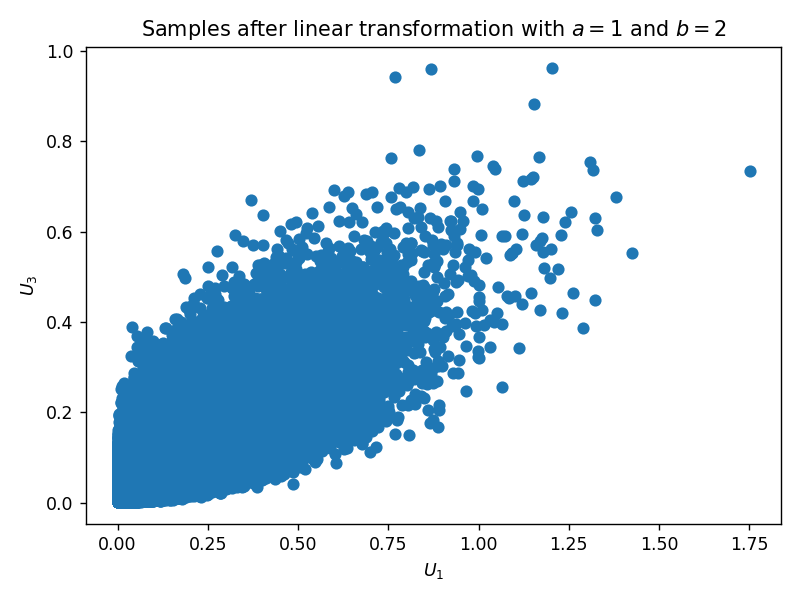}
\end{minipage}%
\begin{minipage}{.33\textwidth}
  \centering
  \includegraphics[width=\linewidth]{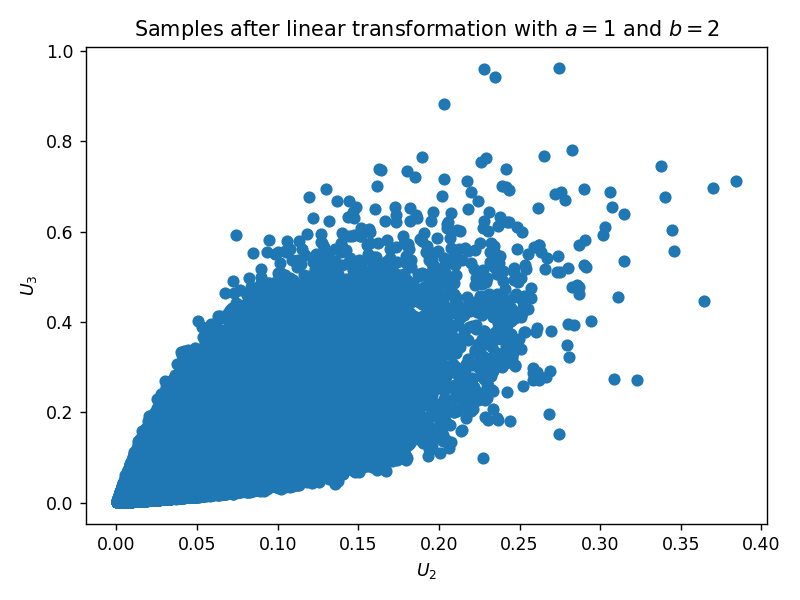}
\end{minipage}

\begin{minipage}{.33\textwidth}
  \centering
  \includegraphics[width=\linewidth]{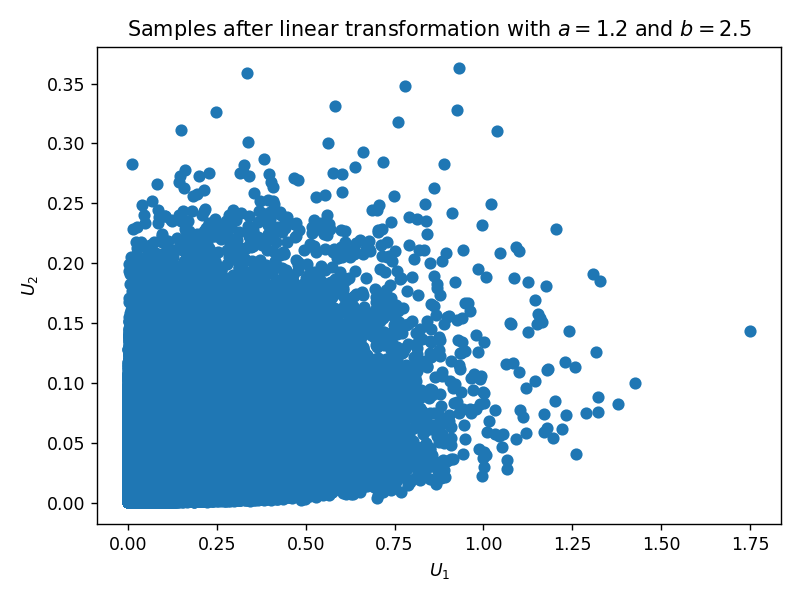}
\end{minipage}%
\begin{minipage}{.33\textwidth}
  \centering
  \includegraphics[width=\linewidth]{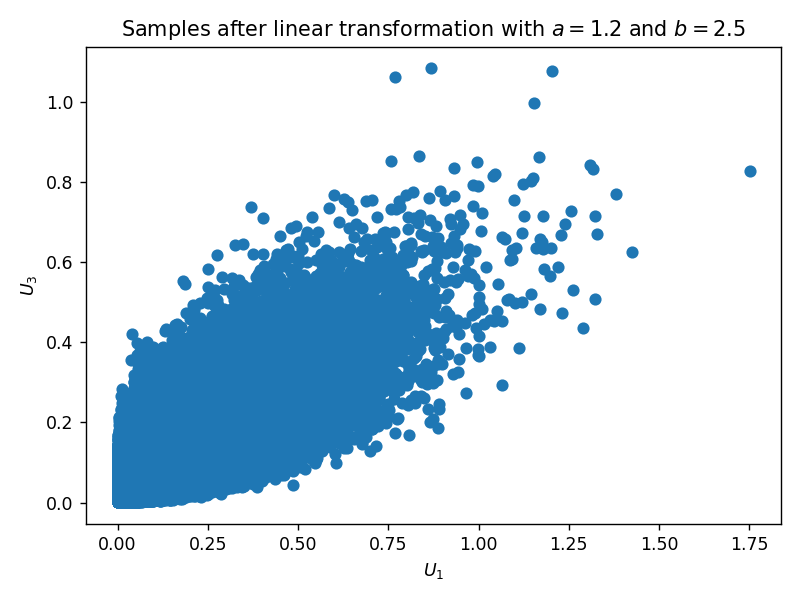}
\end{minipage}%
\begin{minipage}{.33\textwidth}
  \centering
  \includegraphics[width=\linewidth]{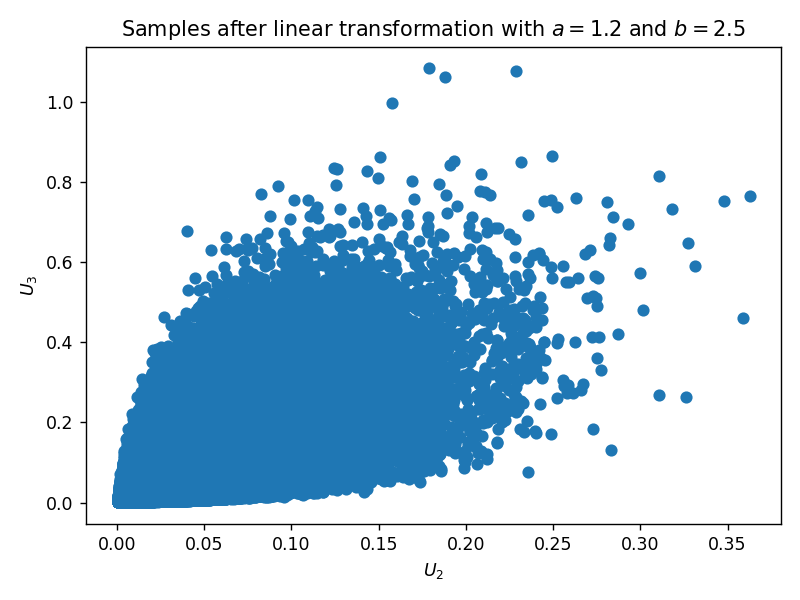}
\end{minipage}

\begin{minipage}{.33\textwidth}
  \centering
  \includegraphics[width=\linewidth]{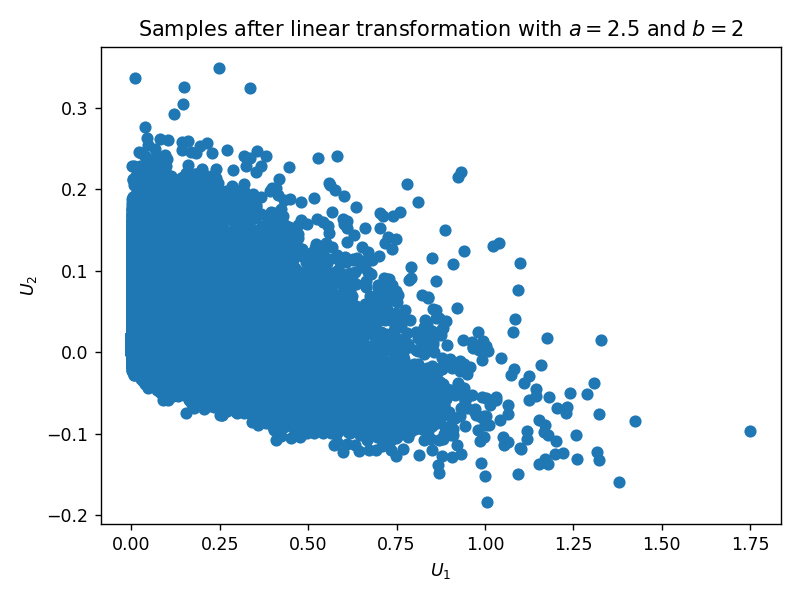}
\end{minipage}%
\begin{minipage}{.33\textwidth}
  \centering
  \includegraphics[width=\linewidth]{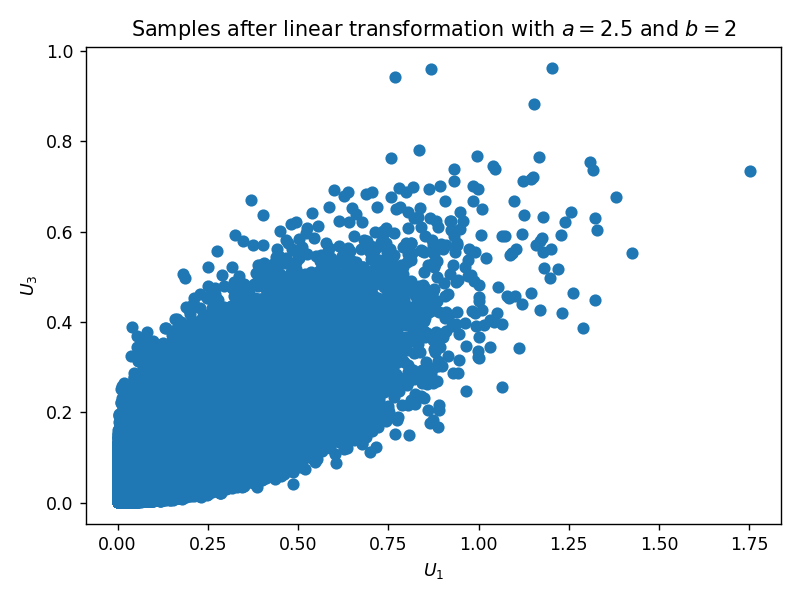}
\end{minipage}%
\begin{minipage}{.33\textwidth}
  \centering
  \includegraphics[width=\linewidth]{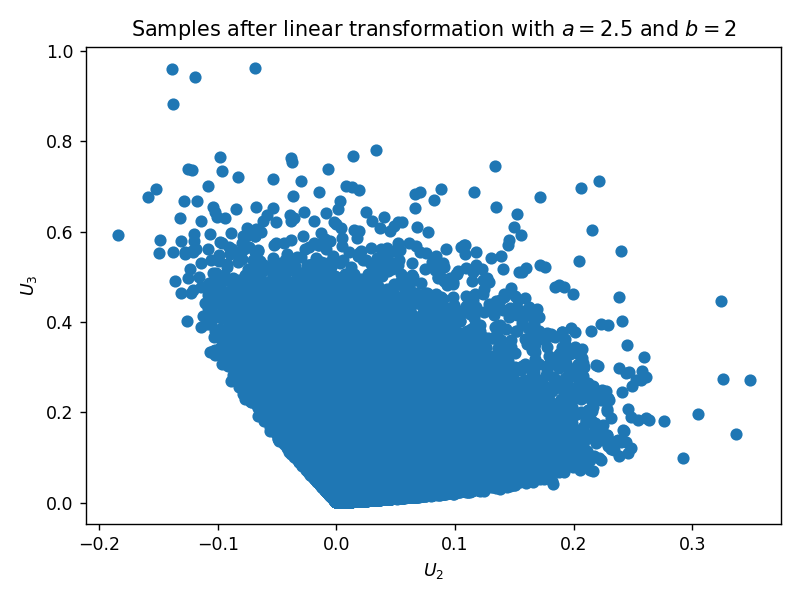}
\end{minipage}
\caption{{Samples of projections of $\bm U =  Q\bm V^3$  for the case of the multifactor square-root process \eqref{eqn:RHestonVolMarkovCompact} with $N=3$,} using $10^3$ sample paths on a time grid with $M=10^5$ time steps. The parameters used are $\bm x = (1, 5, 25), \bm w = (1, 2, 3), \lambda=0.3, \nu=0.3, V_0=0.02, \theta=0.02, T=100,$ and $\bm v_0$ is chosen to be proportional to $\bm x^{-1}$.}
\label{fig:InvariantDomainDimension3}
\end{figure}

\bibliographystyle{plainnat}
\bibliography{literature.bib}

\end{document}